\documentclass[article]{article}
\usepackage[utf8]{inputenc}
\usepackage{amsmath,amsthm,dsfont,pifont,amsfonts,MnSymbol}
\usepackage{hyperref}

\title{$\pi\nos$-spaces and their open images\footnotetext{The work was performed as part of research conducted in the Ural Mathematical Center with the financial support
of the Ministry of Science and Higher Education of the Russian Federation (Agreement number 075-02-2022-874).}}
\author{Vlad Smolin}

\author{Mikhail Patrakeev\footnote{Krasovskii Institute of Mathematics and Mechanics of UB RAS, 620108, 16 Sofia Kovalevskaya street, Yekaterinburg, Russia; \textit{e-mail address}\textup{:} patrakeev@mail.ru}\;,
Vlad Smolin\footnote{Krasovskii Institute of Mathematics and Mechanics of UB RAS, 620108, 16 Sofia Kovalevskaya street, Yekaterinburg, Russia; \textit{e-mail address}\textup{:} SVRusl@yandex.ru}}

\theoremstyle{plain}
\newtheorem{teo}{Theorem}
\newtheorem{lemm}[teo]{Lemma}
\newtheorem{corr}[teo]{Corollary}
\newtheorem{prop}[teo]{Proposition}

\theoremstyle{definition}
\newtheorem{deff}[teo]{Definition}

\newtheorem{nota}[teo]{Notation}

\newtheorem{rema}[teo]{Remark}
\newtheorem{ques}[teo]{Question}

\newcommand{\nos}{\mathsurround=0pt}
\renewcommand{\ll}{\langle}
\newcommand{\rr}{\rangle}
\renewcommand{\leq}{\leqslant}
\renewcommand{\geq}{\geqslant}
\newcommand{\bset}{{}^\omega\hspace{-1pt}\omega}

\newcommand{\bspace}{\omega^\omega}
\newcommand{\btree}{{}^{{<}\omega}\hspace{-1pt}\omega}

\renewcommand{\int}[2]{\mathsf{Int}_{#2}(#1)}
\newcommand{\fruit}[2]{\mathsf{fruit}_{#2}(#1)}
\newcommand{\uph}{{\upharpoonright}\hspace{0.5pt}}

\DeclareMathOperator{\lh}{\mathsf{length}}

\begin{document}

\maketitle

\begin{abstract}
We study spaces that can be mapped onto the Baire space (i.e. the countable power of the countable discrete space) by a continuous quasi-open bijection. 
We give a characterization of such spaces in terms of Souslin schemes and call these spaces $\pi\nos$\emph{-spaces}. 
We show that every space that has a Lusin $\pi\nos$-base is a $\pi\nos$-space and that every second-countable $\pi\nos$-space has a Lusin $\pi\nos$-base. The main result of this paper is a characterization of continuous open images of $\pi\nos$-space.
\end{abstract}

\section{Introduction}

We give a characterization of spaces that can be mapped onto the Baire space by a continuous quasi-open bijection, see Proposition~\ref{another_desc} and Definition~\ref{def_of_pi_sp}. We call such spaces $\pi\nos$\emph{-spaces}. The class of $\pi\nos$\emph{-spaces} is countably productive and the complement of a $\sigma\nos$-compact subset of a $\pi\nos$-space is also a $\pi\nos$-space, see Corollaries~\ref{count_prod} and~\ref{sigma_comp}. 

The class of $\pi\nos$-spaces is similar to the class of spaces with a Lusin $\pi\nos$-base~\cite{patrakeev2015metrizable}:
Every space with a Lusin $\pi\nos$-base is a $\pi\nos$-space and every second-countable $\pi\nos$-space has a Lusin $\pi\nos$-base, see  Remark~\ref{pi.base} and  Theorem~\ref{when_pi_has_lpb}.  
The Baire space and the Sorgenfrey line (and its finite and countable powers) have a Lusin $\pi\nos$-base~\cite{patrakeev2015metrizable,patrakeev2019property}, so these spaces are $\pi\nos$-spaces. 

Every space that has a Lusin $\pi$-base can be mapped onto every nonemty Polish space by a continuous open mapping~\cite{patrakeev2015metrizable}. In Theorem~\ref{example} we show that there exists a $\pi\nos$-space without this property.

In Theorem~\ref{bla_main_theorem} we give a characterization of continuous open images of $\pi\nos$-spaces: For a nonempty topological space ${X}$, the following are equivalent:
\begin{itemize}
  \item[\ding{226}\,] ${X}$ is a continuous open image of a $\pi$-space;
  \item[\ding{226}\,] ${X}$ is a continuous open image of a space that can be mapped onto a Polish space by a continuous quasi-open bijection;
  \item[\ding{226}\,] ${X}$ is a Choquet space of countable $\pi$-weight and of cardinality not greater than continuum.
\end{itemize}

The above theorem implies a characterization of the class of Hausdorff compact continuous open images of $\pi\nos$-spaces: it is the class of Hausdorff compact spaces of countable $\pi$-weight and of cardinality not greater than continuum, see Corollary~\ref{corr.1}. 
And also it implies a characterization of the class of second-countable continuous open images of $\pi\nos$-spaces: it is the class of second-countable Choquet spaces of cardinality not greater than continuum, see Corollary~\ref{corr.2}.

The class of second-countable continuous open images of $\pi\nos$-spaces coincides with the class of second-countable continuous open images of  spaces with a Lusin $\pi\nos$-base, see Corollary~\ref{corr.2}. The question whether the class of all continuous open images of $\pi\nos$-spaces coincides with the class of all continuous open images of spaces with a Lusin $\pi\nos$-base remains open, see Question~\ref{quest}.

\section{Notation and terminology}

We use terminology from~\cite{topenc} and~\cite{kunen2014set}. A \emph{space} is a topological space. We also use the following notation.

\begin{nota}\label{not01}%
  The symbol $\coloneq$ means ``equals by definition''\textup{;}
  the symbol ${\colon}{\longleftrightarrow}$ is used to show that the expression on the left side is an abbreviation for the expression on the right side\textup{;}
  \begin{itemize}
  \item [\ding{46}\,]
 $\mathsurround=0pt
 \omega$ $\coloneq$ the set of finite ordinals $=$ the set of natural numbers, so $0=\varnothing\in\omega$ and ${n}=\{0,\ldots,{n}-1\}$ for all ${n}\in\omega;$
  \item [\ding{46}\,]
 $\mathsurround=0pt
 {s}\,$ is a \textit{sequence}$\ {\colon}{\longleftrightarrow}\ 
 {s}$ is a function whose domain is a finite ordinal or is $\omega$;
\item [\ding{46}\,]
 if ${s}$ is sequence, then

 $\mathsurround=0pt
 \lh({s})\:\coloneq$ the domain of ${s}$;
 \item [\ding{46}\,]
   $\mathsurround=0pt
   \langle{s}_0,\ldots,{s}_{{n}-1}\rangle$ $\coloneq$
   the sequence ${s}$ such that  $\lh({s})={n}\in\omega$ and  ${s}({i})={s}_{i}$ for all ${i}\in{n};$
 \item [\ding{46}\,]
   $\mathsurround=0pt
   \langle\rangle$ $\coloneq$ the sequence of length {0};
 \item [\ding{46}\,]
   if ${s}=\langle{s}_0,\ldots,{s}_{{n}-1}\rangle$, then

   $\mathsurround=0pt
   {s}\hspace{0.5pt}^{\frown}{x}
   \:\coloneq\:\langle{s}_0,\ldots,{s}_{{n}-1},{x}\rangle$;
 \item [\ding{46}\,]
   if ${s}=\langle{s}_0,\ldots,{s}_{{n}-1}\rangle$ and  ${t}=\langle{t}_0,\ldots,{t}_{{m}-1}\rangle$, then

   $\mathsurround=0pt
   {s}\:{\hat{}}\:{t}
   \:\coloneq\:\langle{s}_0,\ldots,{s}_{{n}-1},{t}_0, \ldots, {t}_{{m}-1}\rangle$;
\item [\ding{46}\,]
 $\mathsurround=0pt
 {f}\uph{A}$ $\coloneq$ the restriction of the function ${f}$ to the set ${A};$
 
 \item [\ding{46}\,] ${g}\circ {f}$ is the composition of functions ${g}$ and ${f}$ (that is, ${g}$ after ${f}$);
  \item [\ding{46}\,]
 $\mathsurround=0pt
 {A}\subset{B}
 \ {\colon}{\longleftrightarrow}\ 
 {A}\subseteq{B}\enskip\text{and}\enskip{A}\neq{B};$
  \item [\ding{46}\,]
 if ${s}$ and ${t}$ are sequences, then

 $\mathsurround=0pt
 {s}\sqsubseteq{t}
 \ {\colon}{\longleftrightarrow}\ 
 {s}={t}\uph\lh({s})\quad$ 

 (actually, ${s}\sqsubseteq{t}\leftrightarrow{s}\subseteq{t}$);
  \item [\ding{46}\,]
 $\mathsurround=0pt
 {}^{B}\!{A}$ $\coloneq$ the set of functions from ${B}$ to ${A};$

 in particular, ${}^{0}\hspace{-1pt}{A}=\big\{\langle\rangle\big\};$
\item [\ding{46}\,]
 $\mathsurround=0pt
 {}^{{<}\omega}\hspace{-1pt}{A}
 \:\coloneq\:\bigcup_{{n}\in\omega}{}^{{n}}\hspace{-1pt}{A}\,=\,$
 the set of finite sequences in ${A};$
\item [\ding{46}\,]
 if ${p}$ is a point in a space with topology $\tau$, then

 $\mathsurround=0pt
 \tau({p})$ $\coloneq$ $\{{U}\in\tau:{p}\in{U}\}$ $=$
 the set of open neighbourhoods of ${p}$;
  \item [\ding{46}\,]
 $\mathsurround=0pt
 \gamma\,$ is a $\pi\mathsurround=0pt$-\textit{net} for a space ${X}
 \ {\colon}{\longleftrightarrow}\ $ all elements of $\gamma$ are nonempty and for each nonempty open ${U}\subseteq{X}$, there is ${G}\in\gamma$ such that ${G}\subseteq{U}$;
  \item [\ding{46}\,]
 $\mathsurround=0pt
 \gamma\,$ is a $\pi\mathsurround=0pt$-\textit{base} for a space ${X}
 \ {\colon}{\longleftrightarrow}\ \gamma$ is a $\pi\mathsurround=0pt$-net for ${X}$ and all elements of $\gamma$ are open;
  \item [\ding{46}\,]
  $\mathsurround=0pt
  \tau_{\omega^\omega}$ $\coloneq$  the Tychonoff product topology on the set ${}^{\omega\hspace{-1pt}}\omega$, where $\omega$ carries the discrete topology;
  \item [\ding{46}\,]
  $\mathsurround=0pt
  {\omega^{\omega}}$ $\coloneq$ the \emph{Baire space} $=$ the space $\langle{}^{\omega\hspace{-1pt}}\omega,\tau_{\omega^\omega}\rangle$.
  \end{itemize}
\end{nota}

Recall that, in \cite{kechris2012classical}, a \emph{Souslin scheme}  is an indexed family $\langle V_a \rangle_{a \in \btree}$ of sets. 

\begin{deff}
 Let ${\bf V} = \langle V_a \rangle_{a \in \btree}$ be a Souslin scheme, $\ll{X},\tau\rr$ be a space, and $p \in \bset$. Then

 \begin{itemize}
 \item [\ding{46}\,]  ${\bf V}$  {\it covers} ${X}\ {\colon}{\longleftrightarrow}\  V_{\langle \rangle} = X$ and $V_a = \bigcup_{n \in \omega} V_{a\hspace{0.5pt}^{\frown}n}$ for all $a \in \btree$;
 \item [\ding{46}\,]  ${\bf V}$  {\it partitions} ${X}\ {\colon}{\longleftrightarrow}\  {V}$ covers ${X}$ and  $V_{a\hspace{0.5pt}^{\frown}{n}} \cap V_{a\hspace{0.5pt}^{\frown}{m}} = \emptyset$ for all ${a} \in \btree$ and ${n} \neq {m} \in \omega$; 
 \item [\ding{46}\,] the {\it fruit} of ${p}$ in ${\bf {V}}$, $\fruit{p}{{\bf V}}$, is the set $\bigcap_{n \in \omega}V_{p \upharpoonright n}$;
 \item [\ding{46}\,]  ${\bf V}$ is {\it complete} $\ {\colon}{\longleftrightarrow}\ $  $\fruit{q}{{\bf V}} \neq \emptyset$ for all $q \in \bset$;
 \item [\ding{46}\,]  ${\bf V}$ has {\it strict branches} $\ {\colon}{\longleftrightarrow}\ |\fruit{q}{{\bf V}}|=1$ for all $q \in \bset$;
 \item[\ding{46}\,] ${\bf {V}}$ is {\it open} on $\ll{X},\tau\rr\ {\colon}{\longleftrightarrow}\ {V}_{a}\in\tau$  for all ${a} \in \btree$.
 \end{itemize}
\end{deff}

\begin{deff}\label{def.luz.p.base}
 A \textit{Lusin $\hspace{1pt}\mathsurround=0pt\pi$-base} for a space $\ll{X},\tau\rr$ is an open Souslin scheme $\langle V_a \rangle_{a \in \btree}$ on $X$ that partitions ${X}$, has strict branches, and such that
  \begin{itemize}
 \item[\ding{46}\,]
 $\mathsurround=0pt
   \forall{x}\in{X}\enskip\forall{U}\in\tau({x})$

   $\mathsurround=0pt
   \exists {a}\in\btree\enskip
   \exists{n}\in\omega\ $
   \begin{itemize}
   \item[\ding{226}\,]
  $\mathsurround=0pt
  {x} \in {V}_{a}\enskip$ and
   \item[\ding{226}\,]
  $\mathsurround=0pt
  \bigcup_{{i}\geqslant{n}}{V}_{{a}\hspace{0.5pt}^{\frown}{i}}\subseteq {U}.$
   \end{itemize}
  \end{itemize}
\end{deff}

\begin{nota}\label{not.S.N}\mbox{\ }

  \begin{itemize}
  \item [\ding{46}\,]
  $\mathsurround=0pt
  \mathbf{S}$ $\coloneq$ the \textit{standard Lusin scheme} $\coloneq$ the Souslin scheme $\langle {S}_{a}\rangle_{{a}\in{}^{{<}\omega\hspace{-1pt}}\omega}$ such that

  $\mathsurround=0pt
  {S}_{a}=\{{p}\in{}^{\omega\hspace{-1pt}}\omega:{a}\sqsubseteq{p}\}$
  for all ${a}\in{}^{{<}\omega\hspace{-1pt}}\omega.$
  \end{itemize}
\end{nota}

\begin{rema}\label{rem.baire.space}
  \begin{itemize}
  \item [(a)]
 The family $\{{S}_{a}:{a}\in{}^{{<}\omega\hspace{-1pt}}\omega\}$ is a base for the Baire space.
  \item [(b)]
 The standard Lusin scheme is a Lusin $\pi\mathsurround=0pt$-base for the Baire space.
  \end{itemize}
\end{rema}

\section{$\pi$-spaces}

\begin{deff} \label{def_of_pi_sp}
 A space ${X}$ is a $\pi$-\textit{space} iff there exists an open Souslin scheme $\langle V_a \rangle_{a \in \btree}$ on $X$ that partitions ${X}$, has strict branches, and such that the family $\{{V}_{a} : {a} \in \btree\}$ is a $\pi$-base for ${X}$.
\end{deff}

\begin{rema}\label{pi.base}
  If  $\langle V_a \rangle_{a \in \btree}$ is a Lusin $\pi\mathsurround=0pt$-base for a space ${X},$ then the family $\{{V}_{a}:{a}\in{}^{{<}\omega\hspace{-1pt}}\omega\}$ is a $\pi\mathsurround=0pt$-base for ${X}$.

  It follows that every space with a Lusin $\pi$-base is a $\pi$-space.\hfill$\qed$%
\end{rema}

Recall that a continuous map is \textit{quasi-open} iff the image of every nonempty open set has  nonempty interior.

\begin{lemm} \label{abt_qs_opn}
 For spaces $\ll {X}, \tau \rr$ and $\ll {Y}, \sigma \rr$ the following are equivalent:
 \begin{itemize}
  \item[\textup{(a)}] There exists a continuous quasi-open bijection ${f}\colon \ll {X}, \tau \rr \to \ll {Y}, \sigma \rr$.
  \item[\textup{(b)}] There exists a topology $\rho \supseteq \sigma$ on ${Y}$ such that
   \begin{itemize}
   \item[\ding{226}\,] $\sigma \setminus \{\varnothing\}$ is a $\pi$-base for $\ll Y, \rho \rr$ and 
   \item[\ding{226}\,] $\ll {Y}, \rho \rr$ is homeomorphic to $\ll {X}, \tau \rr$.  \end{itemize}
 \end{itemize}
\end{lemm}

\begin{proof}
 (a) $\Rightarrow$ (b). It is not hard to see that the topology $\rho := \{{f}[{U}] : {U} \in \tau\}$ satisfies all requirements.
 
 (b) $\Rightarrow$ (a). Take a homeomorphism ${f}\colon \ll {X}, \tau \rr \to \ll {Y}, \rho \rr$. Then  ${f}\colon \ll {X}, \tau \rr \to \ll {Y}, \sigma \rr$ is a continuous quasi-open bijection.
\end{proof}

\begin{prop} \label{another_desc}
 For a space ${X}$ the following are equivalent\textup{:}
 \begin{itemize}
  \item[\textup{(a)}] ${X}$ is a $\pi$-space.
  \item[\textup{(b)}] There exists a continuous quasi-open bijection ${f}\colon{X}\to\bspace$.
  \item[\textup{(c)}] There exists a topology $\rho\supseteq\tau_{\bspace}$ on $\bset$ such that
  \begin{itemize}
  \item[\ding{226}\,] $\tau_{\bspace} \setminus \{\varnothing\}$ is a $\pi$-base for $\ll \bset, \rho \rr$ and
  \item[\ding{226}\,] $\ll \bset, \rho \rr$ is homeomorphic to ${X}$.  \end{itemize}
 \end{itemize}
\end{prop}

\begin{proof}
 (b) $\Leftrightarrow$ (c) follows from Lemma \ref{abt_qs_opn}.
 
 (a) $\Rightarrow$ (b). Suppose that $\mathbf{V}=\ll {V}_{a} \rr_{{a} \in \btree}$ is an open Souslin scheme on $X$ that partitions ${X}$ and has strict branches, and such that the family $\{{V}_{a} : {a} \in \btree\}$ is a $\pi$-base for ${X}$. Then for every ${x}\in {X}$,  there is a unique branch ${f}({x}) \in \bset$  such that $\{{x}\} = \fruit{{f}({x})}{{\bf {V}}}$.
 This gives a mapping  ${f}\colon{X}\to\bspace$, and this mapping is bijective. 
 
 Note that ${f}[{{V}_{a}}] = {S}_{a}$ for all ${a}\in\btree$. 
 Then ${X}$ has a $\pi$-base $\{{V}_{a}:{a}\in\btree\}$, whose images $\{{S}_{a}:{a}\in\btree\}$ form a base for $\bspace$ (see Remark~\ref{rem.baire.space}(a)), so the mapping $f$ is quasi-open. 
 And also $\bspace$ has a base $\{{S}_{a}:{a}\in\btree\}$ whose  preimages $\{{V}_{a}:{a}\in\btree\}$ are open, so the mapping $f$ is continuous.
 
 (b) $\Rightarrow$ (a). Suppose that there exists a continuous quasi-open bijection ${f}\colon{X}\to\bspace$. Let   ${V}_{a} \coloneq {f}^{-1}[{S}_{a}]$ for all ${a} \in \btree$. Then $\ll {V}_{a} \rr_{{a} \in \btree}$ is an open Souslin scheme on ${X}$, it partitions ${X}$ and has strict branches, and the family $\{{V}_{a}:{a}\in\btree\}$ is a $\pi$-base for ${X}$.
\end{proof}

Since finite and countable powers of the Baire space are homeomorphic to the Baire space, equivalency of (a) and (c) in Proposition~\ref{another_desc} implies the following.

\begin{corr} \label{count_prod}
Finite and countable products of $\pi\nos$-spaces are also $\pi\nos$-spaces. \hfill$\qed$%
\end{corr}

Since the complement of a $\sigma\nos$-compact subset in the Baire space is homeomorphic to the Baire space~\cite[Theorems 3.11 and 7.7]{kechris2012classical} and since a nonempty open set in the Baire space cannot be covered by a $\sigma\nos$-compact set~\cite[Theorem 7.7]{kechris2012classical}, equivalency of (a) and (c) in Proposition~\ref{another_desc} implies another property of $\pi\nos$-spaces.  

\begin{corr} \label{sigma_comp}
The complement of a $\sigma\nos$-compact subset of a $\pi\nos$-space is a $\pi\nos$-space. \hfill$\qed$%
\end{corr}


\begin{teo} \label{when_pi_has_lpb}
 Every second-countable $\pi$-space has a Lusin $\pi$-base.
\end{teo}

\begin{proof}
 Let ${X}$ be a  second-countable $\pi$-space. By Proposition~\ref{another_desc}, there is a topology $\rho\supseteq\tau_{\bspace}$ on $\bset$ such that
  \begin{itemize}
  \item[\ding{226}\,] $\tau_{\bspace} \setminus \{\varnothing\}$ is a $\pi$-base for $\ll \bset, \rho \rr$ and
  \item[\ding{226}\,] $\ll \bset, \rho \rr$ is homeomorphic to ${X}$.  \end{itemize}

Suppose that $\{{B}_{k} : {k} \text{ is odd}\}$ is a countable base for $\ll \bset , \rho \rr$. We build a  Souslin scheme ${\bf {V}} = \ll {V}_{a} \rr_{{a} \in \btree}$ on $\bset $ that partitions $\bset $ and such that:
 \begin{itemize}
   \item[(a)] $\forall {a} \in \btree\ [{V}_{a} \in \tau_{\bspace} \setminus \{\varnothing\}]$;
   \item[(b)] $\forall {a} \in \btree\ [$ if $\lh({a})$ is odd, then there exists ${d}\in\btree$ such that ${V}_{a} = {S}_{d}$ and $\lh({d}) \geq \lh({a})]$;
   \item[(c)] $\forall {k} \in \omega$ and $\forall {a} \in {}^{k} \omega\ [$ if ${k}$ is odd and ${V}_{a} \cap {B}_{k} \neq \varnothing$, then  $\bigcup_{{n} \geq 1} {V}_{{a}^\frown{n}} \subseteq {B}_{k}]$.
 \end{itemize}

It is easy to check that ${\bf {V}}$ is a Lusin $\pi$-base for $\ll \bset, \rho \rr$.

We build ${\bf {V}}$ by recursion on $\lh({a})$. Let ${V}_{\ll\rr} := \bset $. Suppose we have constructed ${V}_{a}$ for all ${a}\in\btree$ with $\lh({a}) \leq {k}$. Let ${a} \in {}^{k}\omega$; we will define ${V}_{{a}^{\frown}{n}}$ for all ${n} \in \omega$.

Suppose that ${k}$ is even. By condition (a), the set  ${V}_{a}$ is nonempty and open in the Baire space. 
If ${V}_{a}=\bset$, then put ${C}\coloneq\{\ll\rr\}$. If ${V}_{a}\neq\bset$, then put
$$
{C}\coloneq\{{c}
\in\btree:{S}_{c}\subseteq{V}_{a}\text{ and } {S}_{c\upharpoonright (\lh(c)-1)} \not \subseteq {V}_{a}\}.
$$
The family $\{{S}_{c}: {c} \in {C}\}$  is disjoint and ${V}_{a} = \bigcup_{{c} \in {C}}{S}_{c}$. 
Let $${D}\coloneq\{{c}\:{\hat{}}\:{d}:{c}\in{C}\text{ and }{d}\in{}^{k+1}\omega\}.$$ 
The family $\{{S}_{d}: {d} \in {D}\}$ is  disjoint and ${V}_{a}= \bigcup_{{d} \in {D}}{S}_{d}$. Also the set ${D}$ is infinite and  $\lh({d})\geq \lh({a})+1$ for all ${d}\in{D}$. Now we can define the sets ${V}_{{a}^{\frown}{n}}$, ${n}\in\omega$, in such a way that $\{{V}_{{a}^{\frown}{n}}:{n}\in\omega\}=\{{S}_{d}:{d}\in{D}\}$ and all ${V}_{{a}^{\frown}{n}}$ are different.

Suppose that ${k}$ is odd. If ${V}_{a}\cap{B}_{k} = \varnothing$, then we take ${V}_{{a}^{\frown}{n}}$, ${n}\in\omega$, as in the previous case. Suppose that ${V}_{a}\cap{B}_{k}\neq\varnothing$. Since $\tau_{\bspace}\setminus\{\varnothing\}$ is a $\pi$-base for $\ll\bset,\rho\rr$, there is ${c} \in \btree$ such that ${S}_{c}\subset {V}_{a}\cap{B}_{k}$. Let ${V}_{{a}^{\frown}0} := {V}_{a} \setminus {S}_{c}
\neq\varnothing$ and let ${V}_{{a}^{\frown}{n}} := {S}_{{c}^{\frown}({n}-1)}$ for all ${n}\geq 1$.
\end{proof}

From Theorem \ref{when_pi_has_lpb} and \cite[Theorem 3.7]{patrakeev2015metrizable} it follows that

\begin{corr}
If ${X}$ is a second-countable $\pi$-space, then there exists an open continuous map from ${X}$ onto every nonempty Polish space. \hfill$\qed$%
\end{corr}

\begin{teo} \label{example}
 There exists a $\pi$-space that can not be mapped onto the Baire space by a continuous open map. It follows that there exists a $\pi$-space without a Lusin $\pi$-base.
\end{teo}

\begin{proof}
Let $N$ be the family of nowhere dense subsets of the Baire space. Set ${B} := \{{U} \setminus {A} : {U} \in \tau_{\bspace} \text{ and } {A} \in {N}\}$. Since ${N}$ is closed under finite unions, it is easy to see that ${B}$ is closed under finite intersections. Also $B$ covers $\bset$, so $B$ is a base for some topology on $\bset$, which we denote by $\tau_{N}$. 

Let us prove that
 $$
  \ll \bset, \tau_{N} \rr \text{ is a } \pi \text{-space}.
 $$
Using Proposotion \ref{another_desc} it is enough to show that $\tau_{\bspace} \setminus \{\varnothing\}$ is a $\pi$-base for $\ll \bset, \tau_{N} \rr$. Since $\varnothing \in {N}$, it follows that $\tau_{\bspace} \subseteq \tau_{N}$. Take a nonempty ${W} \in \tau_{N}$. There exist ${U} \in \tau_{\bspace} \setminus \{\varnothing\}$ and ${A} \in {N}$ such that ${U} \setminus {A} \subseteq {W}$. Since ${A}$ is nowhere dense in $\bspace$, there exists ${V} \in \tau_{\bspace} \setminus \{\varnothing\}$ such that ${V} \subseteq {U} \setminus {A}$. It follows that ${V} \subseteq {W}$.
 
 Since $\tau_{\bspace} \setminus \{\varnothing\}$ is a $\pi$-base for $\ll \bset, \tau_{N} \rr$, it is easy to see that
 \begin{equation} \label{nwd}
  \forall {B} \subseteq \bset\ [\text{if } {B} \text{ is nowhere dense in } \ll \bset, \tau_{N} \rr \text{, then } {B} \in {N}].
 \end{equation}
 
 Now we prove that
 $$
  \text{there is no continuous open map from } \ll \bset, \tau_{N} \rr \text{ onto the Baire space}.
 $$
 Suppose on the contrary that $f : \ll \bset, \tau_{N} \rr \to \bspace$ is a continuous open surjection. Let ${A}$ be a nowhere dense and not closed subset of the Baire space. Since $f$ is open, it follows that ${f}^{-1}[{A}]$ is nowhere dense in $\ll \bset, \tau_{N} \rr$. Then  ${f}^{-1}[{A}] \in {N}$ by  (\ref{nwd}), so $\bset \setminus {f}^{-1}[{A}] \in \tau_{N}$, and hence $f\big[\bset \setminus {f}^{-1}[{A}]\big]$ is open in  $\bspace$. But $f\big[\bset \setminus {f}^{-1}[{A}]\big]$ equals $\bset\setminus{A}$, so $A$ is closed in $\bspace$, a contradiction.
\end{proof}

Note that the example of a $\pi\nos$-space in this theorem is not a $T_3$-space.

\section{Description of open images of $\pi$-spaces} \label{sec.4}

Recall that the \emph{Choquet game} on a nonempty space $X$ is defined as follows: Two players, I and II, alternately choose nonempty open sets\medskip

I $\ \ \ \ \ \ U_0\qquad\ \ \ \ U_1\qquad\qquad\ldots$

II $\qquad\ \ \ \ \ V_0\qquad\ \ \ \ \  V_1\qquad\qquad\ldots$\medskip\\
such that $U_0\supseteq V_0\supseteq U_1\supseteq V_1 \ldots\ .$ Player~II wins the run  $\big\ll\ll{U}_{0},{V}_{0}\rr,\ll{U}_{1},{V}_{1}\rr,\dots\big\rr$ of Choquet game on $X$ iff $\bigcap_{n}V_n\neq\varnothing;$ otherwise player~I wins this run.
A nonempty space $X$ is called a \emph{Choquet space} iff player~II has a winning strategy in the Choquet game on $X$. More precise definitions of this notions can be found in~\cite{kechris2012classical}.

In this section we will prove the following theorem:

\begin{teo} \label{bla_main_theorem}
 Let ${X}$ be a nonempty space. Then the following are equivalent\textup{:}
 \begin{itemize}
  \item[\textup{(a)}] ${X}$ is a continuous open image of a $\pi$-space.
  \item[\textup{(b)}] ${X}$ is a continuous open image of a space that can be mapped onto a Polish space by a continuous quasi-open bijection.
  \item[\textup{(c)}] ${X}$ is a Choquet space of countable $\pi$-weight and of cardinality not greater than continuum.
 \end{itemize}
\end{teo}

\begin{corr}\label{corr.1}
A Hausdorff compact space is a continuous open image of a $\pi$-space if and only if it has a countable $\pi$-base and its cardinality is not greater than continuum. \hfill \qed 
\end{corr}

\begin{corr}\label{corr.2}
A second-countable space is a continuous open image of a $\pi$-space if and only if it is a continuous open image of a space with a Lusin $\pi$-base if and only if it is a Choquet space of cardinality not greater than continuum.
\end{corr}

\begin{proof}
    We only need to prove that if a second-countable space ${Y}$ is a continuous open image of a $\pi$-space, then it is a continuous open image of a space with a Lusin $\pi$-base. Let $f$ be a continuous open map from a $\pi$-space $\ll {X}, \tau \rr$ onto ${Y}$, $B$ a countable base for ${Y}$, and ${\bf {V}} = \langle {V}_{a} \rangle_{{a} \in \btree}$ an open Souslin scheme on $\ll {X}, \tau \rr$ that partitions ${X}$, has strict branches, and such that the family $\{{V}_{a} : {a} \in \btree\}$ is a $\pi$-base for $\ll {X}, \tau \rr$. Consider the topology $\rho$ on ${X}$ generated by the subbase $\{{V}_{a} : {a} \in \btree\} \cup \{{f}^{-1}[{U}] : {U} \in {B}\}$. It is easy to see that $\ll {X}, \rho \rr$ is a second-countable $\pi$-space and $f \colon \ll {X}, \rho \rr \to {Y}$ is a continuous open surjection. From Theorem \ref{when_pi_has_lpb} it follows that $\ll {X}, \rho \rr$ has a Lusin $\pi$-base.
\end{proof}

\begin{ques}\label{quest}
Do the class of continuous open images of $\pi$-spaces equals the class of continuous open images of spaces with a Lusin $\pi$-base?
\end{ques}

\begin{nota}\label{not.S.N}\mbox{\ } Let ${\bf {V}} = \langle {V}_{a} \rangle_{{a} \in \btree}$ be a Souslin scheme. Then
  \begin{itemize}
 \item [\ding{46}\,] $\mathsf{flesh}({\bf {V}}) = \bigcup_{{a} \in \btree} {V}_{a}$;
 \item [\ding{46}\,]  $\mathsf{branches}_{\bf {V}} ({x}) := \{q \in \bset : x \in \fruit{q}{{\bf V}}\}$.
  \end{itemize}
\end{nota}

\begin{deff}
 A {\it $\pi$-net} Souslin scheme on a space ${X}$ is a Souslin scheme ${\bf {V}} = \langle {V}_{a} \rangle_{{a} \in \btree}$ such that $\mathsf{flesh}({\bf {V}})\subseteq{X}$ and the family $\{{V}_{b}: {b} \sqsupseteq {a}\}$ is a $\pi$-net for the subspace ${V}_{a}$ of ${X}$ for all ${a} \in \btree$.

 A {\it $\pi$-base} Souslin scheme on a space ${X}$ is an open $\pi$-net Souslin scheme on ${X}$.
\end{deff}

\begin{deff}
 A \emph{selector} on a Souslin scheme ${\bf {V}}$ is a surjection  ${f}\colon\bset\to\mathsf{flesh}({\bf {V}})$ such that for all ${x} \in \mathsf{flesh}({\bf {V}})$, the preimage ${f}^{-1}(x)$ is a dense subset of the subspace $\mathsf{branches}_{\bf {V}}({x})$ of the Baire space.
\end{deff}

If a Souslin scheme ${\bf {V}}$ has  strict branches and covers a set ${X}$, then the function ${f} \colon \bset \to {X}$ such that $\{{f}({p})\} = \fruit{p}{{\bf {V}}}$ is a selector on ${\bf {V}}$.

A less trivial example of a selector can be obtained as follows. Let ${f}$ be a continuous surjection from the Baire space onto a space ${X}$. Let ${V}_{a} \coloneq {f}[{S}_{a}]$ for all ${a} \in \btree$. Then ${f}$ is a selector on $\ll {V}_{a} \rr_{{a} \in \btree}$.

\begin{lemm} \label{main_property_of_fiber_scheme}
 Let ${\bf {V}} = \langle {V}_{a} \rangle_{{a} \in \btree}$ be a Souslin scheme that covers $\mathsf{flesh}({\bf {V}})$ and let ${f}$ be a selector on ${\bf {V}}$. Then ${f}[{S}_{a}] = {V}_{a}$ for all ${a} \in \btree$.
\end{lemm}

\begin{proof}
 Let ${a} \in \btree$. We  prove two inclusions.
 
 For $\subseteq$: Let ${p} \in {S}_{a}$. Since ${f}$ is a selector on ${\bf {V}}$ we have ${p} \in \mathsf{branches}_{\bf {V}}({f}({p}))$; that is, ${f}({p}) \in \fruit{p}{{\bf {V}}}$. Also we have $\fruit{p}{{\bf {V}}} \subseteq {V}_{{p} \upharpoonright \lh({a})}$ and ${p} \upharpoonright \lh({a}) = {a}$, therefore ${f}({p}) \in {V}_{a}$.

 For $\supseteq$: Let ${x} \in {V}_{a}$. Since ${\bf {V}}$  covers $\mathsf{flesh}({\bf {V}})$, it follows that ${S}_{a} \cap \mathsf{branches}_{\bf {V}}({x}) \neq \varnothing$. Since ${f}^{-1}({x})$ is dense in $\mathsf{branches}_{\bf {V}}({x})$, there exists ${q} \in ({S}_{a} \cap \mathsf{branches}_{\bf {V}}({x})) \cap {f}^{-1}({x})$, so ${x} = {f}({q}) \in {f}[{S}_{a}]$.
\end{proof}

\begin{deff}
 Let $\langle {X}, \tau \rangle$ be a space, ${\bf {V}}$ a Souslin scheme that covers ${X}$, and ${f}$ a selector on ${\bf {V}}$. Then $\sigma_{\tau, {f}}$ is the topology on $\bset$ generated by the subbase $\{{f}^{-1}[{U}]: {U} \in \tau\}\cup \{{S}_{a}: {a} \in \btree\}$.
\end{deff}

\begin{rema}\label{rema.sigma.tau.f}
 Let $\langle {X}, \tau \rangle$ be a space, ${\bf {V}}$ a Souslin scheme that covers ${X}$, and ${f}$ a selector on ${\bf {V}}$. Then  the family $\{{f}^{-1}[{U}]\cap {S}_{a}:{U} \in \tau,\ {a} \in \btree\}$ is a base for the topology $\sigma_{\tau,{f}}$.  \hfill \qed  
\end{rema}

\begin{lemm} \label{surjection}
 Let ${f} \colon {A} \to {X}$ be a surjection, ${S} \subseteq {A}$, ${V} \subseteq {X}$, and ${f}[{S}] = {V}$. Then ${f}\big[{f}^{-1}[{U}] \cap{S}\big] = {U} \cap {V}$ for all ${U} \subseteq {X}$. \hfill \qed 
\end{lemm}

\begin{lemm} \label{fXV_is_open}
 Let $\langle {X}, \tau \rangle$ be a space, ${\bf {V}}$ an open Souslin scheme on $\langle {X}, \tau \rangle$ that covers ${X}$, and ${f}$ a selector on ${\bf {V}}$. Then ${f} \colon \langle \bset, \sigma_{\tau, {f}} \rangle \to \langle {X}, \tau \rangle$ is a continuous open surjection.
\end{lemm}

\begin{proof}
 It is easy to see that ${f}$ is a continuous surjection. By Remark~\ref{rema.sigma.tau.f}, $\{{f}^{-1}[{U}]\cap {S}_{a}:{U} \in \tau,\ {a} \in \btree\}$ is a base for  $\sigma_{\tau,{f}}$, so it is enough to prove that ${f}\big[{f}^{-1}[{U}]\cap {S}_{a}\big]$ is open for all ${U} \in \tau$ and ${a} \in \btree$. It is true, because  ${f}\big[{f}^{-1}[{U}]\cap {S}_{a}\big] = {U} \cap {V}_{a}$ by Lemmas \ref{surjection} and  \ref{main_property_of_fiber_scheme}.
\end{proof}

\begin{lemm} \label{weqw}
  Let $\langle {X}, \tau \rangle$ be a space, ${\bf {V}}$ a Souslin scheme that covers ${X}$, and ${f}$ a selector on ${\bf {V}}$. Then the weight of the space $\langle \bset, \sigma_{\tau, {f}}\rangle$ is less than or equal to the weight of $\langle {X}, \tau \rangle$. \hfill \qed 
\end{lemm}

\begin{lemm} \label{sigma_tauV_pispace}
 Let $\langle {X}, \tau \rangle$ be a space, ${\bf {V}} = \ll {V}_{a} \rr_{{a} \in \btree}$ a $\pi$-net Souslin scheme on $\ll {X}, \tau \rr$  that covers ${X}$, and ${f}$ a selector on ${\bf {V}}$. Then $\langle \bset,  \sigma_{\tau, {f}}\rangle$ is a $\pi$-space.
\end{lemm}

\begin{proof}
 We prove that condition (c) in Proposition~\ref{another_desc} is satisfied for $\rho=\sigma_{\tau,f}$. Obviously, $\sigma_{\tau,{f}}\supseteq\tau_{\bspace}$. By Remark~\ref{rema.sigma.tau.f}, $\{{f}^{-1}[{U}]\cap {S}_{a}:{U} \in \tau,\ {a} \in \btree\}$ is a base for  $\sigma_{\tau,{f}}$, so it is enough to prove that for all ${U} \in \tau$ and ${a} \in \btree$ with nonempty ${f}^{-1}[{U}]\cap {S}_{a}$, there exists ${b} \in \btree$ such that ${S}_{b} \subseteq {f}^{-1}[{U}]\cap {S}_{a}$.

 From Lemmas \ref{surjection} and  \ref{main_property_of_fiber_scheme} it follows that $f\big[{f}^{-1}[{U}]\cap {S}_{a}\big] = {U} \cap {V}_{a}$,  so ${U} \cap {V}_{a}$ is nonempty open set in the subspace ${V}_{a}$ of $\ll{X},\tau\rr$. Then  there exists ${b} \in \btree$ such that ${b} \sqsupseteq {a}$ and ${V}_{b} \subseteq {U} \cap {V}_{a}$ because ${\bf {V}}$ is a $\pi$-net Souslin scheme on $\ll {X}, \tau \rr$. From Lemma \ref{main_property_of_fiber_scheme} it follows that ${S}_{b} \subseteq {f}^{-1}[{V}_{b}] \subseteq {f}^{-1}[{U}]$, and since ${b} \sqsupseteq {a}$, we see that ${S}_{b} \subseteq {S}_{a}$. Therefore ${S}_{b} \subseteq {f}^{-1}[{U}]\cap {S}_{a}$. 
\end{proof}

\begin{lemm} \label{lemm_perm}
 Let ${\bf {V}} = \langle {V}_{a} \rangle_{{a} \in \btree}$ be a Souslin scheme and let ${g} \colon \omega \to \omega$ be a surjection. Then
 \begin{itemize}
 
\item[\textup{(a)}] $\{{g}\circ({a}\:{\hat{}}\:{c}):{c}\in\btree\}=\{({g}\circ{a})\:{\hat{}}\:{c}:{c}\in\btree\}$ for all ${a}\in\btree$;

\item[\textup{(b)}]  $\bigcup_{{n} \in \omega}V_{({g}\circ {a})\hspace{0.5pt}^{\frown}{n}} = \bigcup_{{n} \in \omega}{V}_{{g}\circ ({a}\hspace{0.5pt}^{\frown}{n})}$ for all ${a} \in \btree$;

\item[\textup{(c)}]  $\bigcap_{{n} \in \omega} {V}_{{g}\circ ({q}\uph{n})} = \bigcap_{{n} \in \omega} {V}_{({g}\circ {q}) {\upharpoonright} {n}}$ for all ${q} \in \bset$.\hfill \qed 
 \end{itemize}
\end{lemm}

\begin{lemm} \label{blabla}
  Let ${\bf {V}} = \langle {V}_{a} \rangle_{{a} \in \btree}$ be a Souslin scheme and ${x} \in \mathsf{flesh}({\bf {V}})$. Suppose that for all ${a} \in \btree$, if ${V}_{a}\ni{x} $, then there are ${n} \neq {m} \in \omega$ such that $ {V}_{{a}\hspace{0.5pt}^{\frown}{n}}\ni{x}$ and ${V}_{{a}\hspace{0.5pt}^{\frown}{m}}\ni{x}$. Then $\mathsf{branches}_{\bf {V}}({x})$ is a dense-in-itself subspace of the Baire space.\hfill \qed 
\end{lemm}

\begin{nota}
 Suppose that ${\bf {V}} = \langle {V}_{a} \rangle_{{a} \in \btree}$ is a Souslin scheme and ${g} \colon \omega \to \omega$. Then ${\bf {V}}^{{g}} = \langle {V}^{{g}}_{a} \rangle_{{a} \in \btree}$ is a Souslin scheme such that ${V}^{{g}}_{a} \coloneq {V}_{{g} \circ {a}}$ for all ${a} \in \btree$.
\end{nota}

\begin{lemm}\label{prop_of_Pfv}
 Let ${\bf {V}} = \langle {V}_{a} \rangle_{{a} \in \btree}$ be a Souslin scheme, ${g} \colon \omega \rightarrow \omega$ a surjection, and ${X}$ a space. Then:
 \begin{itemize}
  \item[\textup{(a)}]  If $\,{\bf {V}}\!$ covers ${X}$, then ${\bf {V}}^{{g}}$ covers ${X}$.

  \item[\textup{(b)}]  If $\,{\bf {V}}\!$ is complete, then ${\bf {V}}^{{g}}$ is complete.

  \item[\textup{(c)}]  If $\,{\bf {V}}\!$ is open on ${X}$, then ${\bf {V}}^{{g}}$ is open on ${X}$.

  \item[\textup{(d)}]  If $\,{\bf {V}}\!$ covers ${X}$ and $|{g}^{-1}({n})| \geq 2$ for all ${n} \in \omega$, then $\mathsf{branches}_{{\bf {V}}^{{g}}}({x})$ is a dense-in-itself subspace of the Baire space for all ${x} \in {X}$.

  \item[\textup{(e)}] If $\,{\bf {V}}\!$ is a $\pi$-net Souslin scheme on ${X}$, then ${\bf {V}}^{{g}}$ is a $\pi$-net Souslin scheme on ${X}$.
 \end{itemize}
\end{lemm}

\begin{proof}
 For (a), let ${a} \in \btree$. From (b) of Lemma \ref{lemm_perm} it follows that ${V}^{{g}}_{a} = {V}_{{g}\circ {a}} = \bigcup_{{n} \in \omega}{V}_{
 ({g}\circ {a})\hspace{0.5pt}^{\frown}{n}} = \bigcup_{{n} \in \omega}{V}_{{g}\circ ({a}\hspace{0.5pt}^{\frown}{n})} = \bigcup_{{n} \in \omega}{V}^{g}_{{a}\hspace{0.5pt}^{\frown}{n}}$.

 For (b), let ${q} \in \bset$. From (c) of Lemma \ref{lemm_perm} it follows that 
 $$
 \fruit{q}{{\bf V}^{g}} = \bigcap_{{n} \in \omega} {V}^{g}_{{q} {\upharpoonright} {n}} = \bigcap_{{n} \in \omega} {V}_{{g}\circ ({q} {\upharpoonright} {n})} = \bigcap_{{n} \in \omega} {V}_{({g}\circ {q}) {\upharpoonright} {n}} = 
 \fruit{{g}\circ{q}}{{\bf V}}\neq \emptyset.
 $$

 (c) is obvious. For (d), let ${x} \in {X}$. Using Lemma \ref{blabla} it is enough to show that for all ${a} \in \btree$, if ${V}^{g}_{a}\ni{x}$, then there are ${n} \neq {m} \in \omega$ such that ${V}^{g}_{{a}\hspace{0.5pt}^{\frown}{n}}\ni{x}$ and ${V}^{g}_{{a}\hspace{0.5pt}^{\frown}{m}}\ni{x}$. Take ${a} \in \btree$ such that ${V}^{g}_{a}\ni{x}$. By (a), ${\bf {V}}^{g}$  covers ${X}$, so there exists ${n} \in \omega$ such that ${x} \in {V}^{g}_{{a}\hspace{0.5pt}^{\frown}{n}}$. Take ${m} \in \omega \setminus \{{n}\}$ such that ${g}({m}) = {g}({n})$. Then ${x} \in {V}^{g}_{{a}\hspace{0.5pt}^{\frown}{n}} = {V}_{{g}\circ ({a}\hspace{0.5pt}^{\frown}{n})} = {V}_{{g}\circ ({a}\hspace{0.5pt}^{\frown}{m})} = {V}^{g}_{{a}\hspace{0.5pt}^{\frown}{m}}$.

 For (e), let ${a} \in \btree$. We must show that  $\{{V}^{g}_{b}:{b}\sqsupseteq{a}\}$ is a $\pi$-net for the subspace ${V}^{g}_{a}$ of ${X}$. From (a) of Lemma \ref{lemm_perm} it follows that 
 $$
 \{{V}^{g}_{b}:{b}\sqsupseteq{a}\}=
 \{{V}^{g}_{{a}\:{\hat{}}\:{c}}:{c}\in \btree\}=
 \{{V}_{{g}\circ ({a}\:{\hat{}}\:{c})}:{c}\in\btree\} =
 $$
 $$
 =\{{V}_{({g}\circ {a})\:{\hat{}}\:{c}}:{c}\in\btree\} = 
 \{{V}_{b}:{b}\sqsupseteq{g}\circ{a}\}.
 $$
 Since ${\bf {V}}$ is a $\pi$-net Souslin scheme on ${X}$, we see that $\{{V}_{b}:{b}\sqsupseteq{g}\circ{a}\}$ is a $\pi$-net for the subspace ${V}_{{g}\circ {a}}={V}^{g}_{a}$ of ${X}$.
\end{proof}

\begin{lemm} \label{fiber_scheme}
 Let ${\bf {V}}$ be a complete Souslin scheme that covers $\mathsf{flesh}({\bf {V}})$ and  $|\mathsf{flesh}({\bf {V}})| \leq 2^{\aleph_0}$. Suppose that  $\mathsf{branches}_{\bf {V}}({x})$ is a dense-in-itself subspace of the Baire space for all ${x} \in \mathsf{flesh}({\bf {V}})$. Then ${\bf {V}}$ has a selector.
\end{lemm}

\begin{proof}
 Note that, since $\mathbf{V}$ covers $\mathsf{flesh}({\bf {V}})$, $\mathsf{branches}_{\bf {V}}({x})$ is nonempty for all ${x} \in \mathsf{flesh}({\bf {V}})$.
 
 Note also that if ${x}\in\mathsf{flesh}({\bf {V}})$ and ${U}$ is nonempty open set in the subspace $\mathsf{branches}_{\bf {V}}({x})$ of the Baire space, then ${U}$ has cardinality $2^{\aleph_0}$. Indeed, $U$ is open in $\mathsf{branches}_{\bf {V}}({x})$ and $\mathsf{branches}_{\bf {V}}({x})$ is closed in $\bspace$, so ${U}$ is Polish. Also ${U}$ is dense-in-itself because $\mathsf{branches}_{\bf {V}}({x})$ is dense-in-itself. Therefore $|{U}|=2^{\aleph_0}$. 
 
 Now, by transfinite recursion on $\mathsf{flesh}({\bf {V}})$ well-ordered in the type of its cardinality, it is easy to build an indexed family $\langle {Q}_{x} \rangle_{{x} \in \mathsf{flesh}({\bf {V}})}$ such that
 \begin{itemize}
 \item[\ding{226}\,] ${Q}_{x}$ is a dense subset of $\mathsf{branches}_{\bf {V}}({x})$ for all ${x} \in \mathsf{flesh}({\bf {V}})$ and
  \item[\ding{226}\,] ${Q}_{x} \cap {Q}_{y} = \emptyset$ for all ${x} \neq {y} \in \mathsf{flesh}({\bf {V}})$.
 \end{itemize}

 Now we can construct a selector ${f}\colon\bset\to\mathsf{flesh}({\bf {V}})$ on ${\bf {V}}$. If ${p}
 \in{Q}_{x}$ for some ${x}\in\mathsf{flesh}({\bf {V}})$, then set ${f}({p})\coloneq{x}$.  If ${p}\nin\bigcup_{{x}\in\mathsf{flesh}({\bf {V}})}{Q}_{x}$, then choose ${f}({p})\in\fruit{p}{{\bf {V}}}$ arbitrarily. It is easy to see that ${f}$ is a selector on ${\bf {V}}$.
\end{proof}


 

\begin{lemm} \label{pibasesch_fibpibasesch}
 Let ${X}$ be a space and $|{X}| \leq 2^{\aleph_0}$. Suppose that there exists a complete $\pi$-base Souslin scheme on ${X}$ that covers ${X}$. Then there exists a complete $\pi$-base Souslin scheme on ${X}$ that covers ${X}$ and has a selector.
\end{lemm}

\begin{proof}
 Let ${\bf {V}}$ be a complete $\pi$-base Souslin scheme on ${X}$ that covers ${X}$. Consider a surjection ${g}\colon \omega \to \omega$ such that $|{g}^{-1}({n})| \geq 2$ for all ${n} \in \omega$. From Lemma \ref{prop_of_Pfv} it follows that ${\bf {V}}^{{g}}$ is a complete $\pi$-base Souslin scheme on ${X}$ that covers ${X}$, and that ${\bf {V}}^{{g}}$ meets the premisses of Lemma \ref{fiber_scheme}. Therefore ${\bf {V}}^{{g}}$ has a selector.
\end{proof}

\begin{lemm} \label{suff}
 Suppose that there exists a $\pi$-base Souslin scheme on a space ${X}$ that covers ${X}$ and has a selector. Then ${X}$ is a continuous open image of a $\pi$-space.
\end{lemm}

\begin{proof}
 This lemma follows from Lemmas \ref{fXV_is_open} and \ref{sigma_tauV_pispace}.
\end{proof}



\begin{lemm} \label{desc_op_im}
 A space ${X}$ is a continuous open image of a $\pi$-space if and only if $|{X}| \leq 2^{\aleph_0}$ and there exists a complete $\pi$-base Souslin scheme on ${X}$ that covers ${X}$.
\end{lemm}

\begin{proof}
For $\Rightarrow$: Let ${f}$ be a continuous open surjection from a $\pi$-space ${Y}$ onto ${X}$. Let $\mathbf{V}=\ll {V}_{a} \rr_{{a} \in \btree}$ be an open Souslin scheme on $Y$ that partitions ${Y}$, has strict branches, and such that the family $\{{V}_{a} : {a} \in \btree\}$ is a $\pi$-base for ${Y}$. It is easy to see that $\ll {f}[{V}_{a}] \rr_{{a} \in \btree}$ is a complete  $\pi$-base Souslin scheme on ${X}$ that covers ${X}$.

For $\Leftarrow$: From Lemma \ref{pibasesch_fibpibasesch} it follows that there exists $\pi$-base Souslin scheme on ${X}$ that covers ${X}$ and has a selector. So ${X}$ is a continuous open image of a $\pi$-space by Lemma \ref{suff}.
\end{proof}

Let ${X}$ be a space and ${s}=\big\ll\ll{U}_{0},{V}_{0}\rr,\dots,\ll{U}_{n},{V}_{n}\rr\big\rr$ a sequence of moves in the Choquet game on ${X}$. For a  positive ${k}\leqslant{n}$, we call a pair $\ll{U}_{k},{V}_{k}\rr$ of the sequence ${s}$ \emph{redundant} iff ${V}_{k}={U}_{k}={V}_{{k}-1}$. 
Also we call the pair $\ll{U}_0,{V}_0\rr$ \emph{redundant} iff ${V}_{0}={U}_{0}={X}$. Let ${f}({s})$ be the sequence that is obtained by removing all redundant pairs from ${s}$. For example, if ${Z}\subset{Y}\subset{X}$, then
$$
{f}\big(\big\ll \ll{X},{X}\rr,\ll{X},{X}\rr,\ll{Y},{Y}\rr,\ll{Y},{Y}\rr,\ll{Z},{Z}\rr \big\rr\big) = \big\ll\ll{Y},{Y}\rr,\ll{Z},{Z}\rr\big\rr.
$$

Let $\Gamma$ be a strategy for player II in the Choquet game on a space ${X}$. The \emph{modification} $\Gamma'$ of the strategy $\Gamma$ prescribes player II to play as follows. Let  
$$
{s}=\big\ll\ll{U}_{0},{V}_{0}\rr,\dots,\ll{U}_{{n}},{V}_{{n}}\rr\big\rr
$$ 
be a sequence of previous moves. 
If in the $({n}{+}1\hspace{-1pt})\mathsurround=0pt$th move player I plays ${U}_{{n}+1}={V}_{{n}}$, then $\Gamma'({s},{U}_{{n}+1})\coloneq{V}_{n}$; that is, $\Gamma'$ tells player II to reply with the set ${V}_{{n}+1}\coloneq{V}_{n}$. 
If player I plays ${U}_{{n}+1}\neq{V}_{n}$, then $\Gamma'({s},{U}_{{n}+1})\coloneq\Gamma\big({f}({s}),{U}_{{n}+1}
\big)$. 
Also we set $\Gamma'(\ll\rr,{X})\coloneq{X}$ and $\Gamma'(\ll\rr,{U}_{0})\coloneq\Gamma(\ll\rr,{U}_{0})$ for  ${U}_{0}\neq{X}$.



\begin{lemm}\label{modified}
 If $\,\Gamma\!$ is a winning strategy for player II in the Choquet game on a space ${X}$, then its modification $\,\Gamma'\!$ is also winning.\hfill\qed
\end{lemm}

\begin{proof}[Proof of Theorem~\ref{bla_main_theorem}]
This theorem says that for every nonempty space ${X}$, the following are equivalent:
 \begin{itemize}
  \item[\textup{(a)}] ${X}$ is a continuous open image of a $\pi$-space.
  \item[\textup{(b)}] ${X}$ is a continuous open image of a space that can be mapped onto a Polish space by a continuous quasi-open bijection.
  \item[\textup{(c)}] ${X}$ is a Choquet space of countable $\pi$-weight and of cardinality not greater than continuum.
 \end{itemize}

(a)$\Rightarrow$(b) follows from Proposition \ref{another_desc}.

(b)$\Rightarrow$(c). Suppose that ${X}$ is a continuous open image of a space ${Y}$ that can be mapped onto a Polish space by a continuous quasi-open bijection. It is easy to show that ${Y}$ is a Choquet space of countable $\pi$-weight and of cardinality not greater than continuum. Then ${X}$ also possesses these properties.

(c)$\Rightarrow$(a).  Using Lemma \ref{desc_op_im}, it is enough to find a complete  $\pi$-base Souslin scheme on ${X}$ that covers ${X}$. 
Let $\Gamma$ be a winning strategy for player II in the Choquet game on ${X}$.

We will build Sousin schemes ${\bf {U}} = \ll {U}_{a} \rr_{{a} \in \btree}$ and ${\bf {V}} = \ll {V}_{a} \rr_{{a} \in \btree}$ such that
\begin{itemize}
 \item[\ding{226}\,]  ${\bf {V}}$ covers ${X}$;
 \item[\ding{226}\,]  ${V}_{a}$ is nonempty and open for all ${a} \in \btree$;
 \item[\ding{226}\,]  $\{{V}_{a^{\frown}{m}} : {m} \in \omega\}$ is a $\pi$-base for the subspace ${V}_{a}$ of ${X}$ for all ${a} \in \btree$;
 \item[\ding{226}\,]  for every branch ${p} \in \bset$, the sequence $\big\ll \ll{U}_{p\uph{0}}, {V}_{p\uph{0}}\rr, \ll{U}_{p\uph{1}}, {V}_{p\uph{1}}\rr, \dots \big\rr$ is a run of the Choquet game on ${X}$ in which player II plays according to the modified strategy $\Gamma'$.
\end{itemize}
Lemma~\ref{modified} says that $\Gamma'$ is a winning strategy for player II in the Choquet game on ${X}$, so the scheme ${\bf {V}}$ is  complete. Also it is easy to see that ${\bf {V}}$ is a $\pi$-base Souslin scheme on ${X}$ that covers ${X}$. 

To compete the proof, it remains to construct ${\bf {U}} = \ll {U}_{a} \rr_{{a} \in \btree}$ and ${\bf {V}} = \ll {V}_{a} \rr_{{a} \in \btree}$; we do it by recursion on $\lh({a})$. 
Put ${U}_{\ll \rr} \coloneq {X}$ and
${V}_{\ll \rr} := \Gamma'({\ll\rr},{U}_{\ll \rr})$. Note that 
${V}_{\ll \rr}={X}$. 

Suppose that ${U}_{a}$ and ${V}_{a}$ have been chosen; we will chose ${U}_{a^{\frown}{m}}$ and ${V}_{a^{\frown}{m}}$ for all ${m}\in\omega$. 
Consider a sequence $\ll {W}_{m} \rr_{{m} \in \omega}$ such that  ${W}_{0} = {V}_{a}$ and $\{{W}_{m} : {m} \in \omega\}$ is a $\pi$-base for the subspace ${V}_{a}$ of ${X}$. For all ${m}\in\omega$, we put ${U}_{a^{\frown}{m}} \coloneq {W}_{m}$ and 
$$
{V}_{a^{\frown}{m}} := \Gamma'\big(\big\ll\ll{U}_{{a}\uph{0}}, {V}_{{a}\uph{0}}\rr,\ll{U}_{{a}\uph{1}},{V}_{{a}\uph{1}}\rr, \dots, \ll{U}_{{a}},{V}_{{a}}\rr \big\rr,{U}_{a^{\frown}{m}}\big).
$$
\end{proof}


\end{document}